\documentclass{article}
\usepackage[utf8]{inputenc}

\usepackage[top=3cm,bottom=2cm,left=2cm,right=2cm,marginparwidth=1.75cm]{geometry}
\usepackage{amsmath,amsthm,amssymb,amscd,esint,enumitem}
\usepackage{authblk}
\usepackage{enumerate}
\usepackage{tikz}
\usepackage{parskip}
\usepackage{url}
\usepackage{xcolor}
\usepackage[colorlinks,citecolor=red,pagebackref,hypertexnames=false]{hyperref}

\usepackage{latexsym}
\usepackage{bm}
\usepackage{mathtools}

\newcommand{\surjects}{\twoheadrightarrow}

\newcommand{\R}{\mathbb{R}}

\newcommand{\C}{\mathbb{C}}

\newcommand{\PRD}{\textbf{P}\mathbb{R}^{{\binom{D+2}{2}}}}
\newcommand{\PCD}{\textbf{P}\mathbb{C}^{\binom{D+2}{2}}}

\newcommand{\e}{\varepsilon}

\newcommand{\f}{\varphi}

\newtheoremstyle{problemstyle}  
{3pt}                                               
{3pt}                                               
{\normalfont}                               
{}                                                  
{\bfseries}                 
{\bfseries .}         
{.5em}                                          
{}                                                  
\theoremstyle{problemstyle}

\newtheorem{thm}{Theorem}[section]
\newtheorem{lem}[thm]{Lemma}
\newtheorem{dfn}[thm]{Definition}

\title{Distinct Distances Between a Circle and a Generic Set}
\author{Alex McDonald, Brian McDonald, Jonathan Passant and Anurag Sahay}
\date{\today}

\begin{document}

\maketitle

\begin{abstract}
	Let $S$ be a set of points in $\R^2$ contained in a circle and $P$ an unrestricted point set in $\R^2$. We prove the number of distinct distances between points in $S$ and points in $P$ is at least 	$\min(|S||P|^{1/4-\e},|S|^{2/3}|P|^{2/3},|S|^2,|P|^2)$.
	This builds on work of Pach and De Zeeuw \cite{PdZ17}, Bruner and Sharir \cite{BS18}, McLaughlin and Omar \cite{MO18} and Mathialagan \cite{M19} on distances between pairs of sets.
\end{abstract}

\section{Introduction}

In 1945 Erd\H os introduced his distinct-distances problem, first stated in \cite{E45}, asking for the minimum number of distinct distances an $n$ point set can create  in $\mathbb{R}^2$. Erd\H os showed that a square lattice $\Lambda$ of $n$ points determined $|\Delta(\Lambda)|\lesssim \frac{n}{\sqrt{\log n}}$ distances, where here and throughout, $\gtrsim $ and $\lesssim$ are used to suppress some constant independent of the controlling parameter and $\Delta(S)$ denotes the set of distances between elements of $S$. Erd\H os conjectured that this was essentially the best possible i.e. for any $n$ point set $P$ one has $|\Delta(P)|\gtrsim n^{1-\varepsilon}$, for all $\varepsilon>0$. This question played a key role in the combinatorial geometry for over 50 years, with many successive improvements, see \cite[Section 5.3]{BMP05} or \cite{GIS11}. In \cite{GK15} Guth and Katz provided a solution utilising significant new algebraic developments in what has become known as the polynomial method in combinatorics. \\

A natural variant of the distinct distances problem asks for the minimum number of distances between points $a\in A,b\in B$ where one or both of the finite sets $A$ and $B$ are constrained in some way.  Purdy \cite[Section 5.5]{BMP05} considers a version of this problem where $\ell_1$ and $\ell_2$ are lines in the plane, and $A$ and $B$ consist of $n$ points on $\ell_1$ and $\ell_2$, respectively.  He observed that if $\ell_1$ and $\ell_2$ are parallel or perpendicular then one may imitate the grid example given by Erd\H os to obtain $\lesssim n$ distances, but conjectured that otherwise the number of distnaces was superlinear.  More precisely, he conjectured that for every $C$ there exists $n_0$ such that if $A\subset \ell_1,B\subset\ell_2$ each have size $n>n_0$ and determine $<Cn$ distances, then $\ell_1$ and $\ell_2$ are either parallel or perpendicular. Elekes and R\'onyai \cite{ER00} prove this conjecture from a statement about restricted polynomials, implicitly showing that there exists $\delta>0$ such that $\Delta(A,B)\gtrsim n^{1+\delta}$ under the conditions of the conjecture. Elekes \cite{E99} subsequently showed that one can in fact take $\delta=\frac{1}{4}$. Schwartz, Solymosi, and de Zeeuw prove that in the unbalanced version of the problem, where $|A|=n^{\frac{1}{2}+\e}$ and $|B|=n$, the number of distances is still superlinear.  The results in both the balanced and unbalanced cases were improved by Sharir, Sheffer and Solymosi \cite{SSS13} who use algebraic techniques to show that sets $A$ and $B$ of size $n$ and $m$, respectively, satisfying the hypotheses above, determine $\gtrsim\min\{n^{2/3}m^{2/3}, n^2, m^2\}$ distances. \\

This question has been generalised in two key ways. Pach and de Zeeuw \cite{PdZ17} showed that if you restrict your point sets to two algebraic curves $C_1$ and $C_2$ of constant degree (constant with respect to $n$ and $m$ the number of points on $C_1$ and $C_2$ respectively), then one has at least $\gtrsim \min\{n^{2/3}m^{2/3}, n^2, m^2\}$ distinct distances, provided the curves are not parallel lines, orthogonal lines, or concentric circles. This argument used heavily that both point sets are on curves, so their role could be interchanged and that curves that are not  parallel lines, orthogonal lines, or concentric circles don't share too many symmetries. \\

In a different direction Bruner and Sharir \cite{BS18} showed that when the first set $P$ of $m$ points is on a line and the second $P'$ of $n$ points is unrestricted in the plane one has at least

$$ |\Delta(P,P')|  \gtrsim \min\left\{n^{2/3}m^{2/3}, \frac{m^{10/11}n^{4/11}}{\log^{2/11}n}, n^2,m^2\right\}.$$ 

This result relied on the explicit parametrisation of the line to build an incidence problem. Similarly McLaughlin and Omar \cite{MO18} showed that if you restrict $m$ points $P$ to a curve of constant degree and have $n$ unrestricted points $P'$ one has

$$ |\Delta(P,P')| \gtrsim \begin{cases}m^{1/2}n^{1/2}\log^{-1/2}n & \text{when } m \gtrsim n^{1/2}\log^{-1/3}n,\\
m^{1/3}n^{1/2} & \text{when } m \lesssim n^{1/2}\log^{-1/3}n.\end{cases}$$

Finally, Mathialagan \cite{M19} extended these results in $\R^2$ to the setting where $P$ and $P'$ are both unrestricted point sets (of size $m$ and $n$ respectively) 

$$ |\Delta(P,P')| \gtrsim \begin{cases} m^{1/2}n^{1/2} \log^{-1} n & \text{when } n^{1/3} \leq m \leq n, \\ m^{1/2}n^{1/2} & \text{when } m\leq n^{1/3} \end{cases} $$

Since $P$ and $P'$ are symmetric in this case, analogous bounds hold when $m \geq n$. In particular, this subsumes McLaughlin-Omar's result.  In the balanced case ($n^{1/3}\leq m\leq n$), Mathialagan's result is obtained by adapting the Guth-Katz argument \cite{GK15} to the question of distances determined by two sets (rather than one).  In the unbalanced case ($2\leq m\leq n^{1/3}$) Mathialagan shows that in fact there are $\gtrsim (nm)^{1/2}$ pinned distances determined by the bigger set and a single point in the smaller set.  It is interesting to compare this result to those which assume algebraic structure on one or both of their sets.  In our result, as well as the results of Bruner-Sharir and Pach-de Zeeuw, the bounds obtained are better than Mathialagan's in the near-balanced case (how close to balanced the sets must be depends on the particular result) but in extremely lopsided cases the $(mn)^{1/2}$ bound eventually wins. \\

The aim of this paper is to find similar results in the case where one set lies on a circle and the second set is essentially unrestricted. We obtain the following result.

\begin{thm}\label{Thm: MainResult}Suppose that $S$ is a point set on a circle in $\mathbb{R}^2$ and $P$ is a point set in $\mathbb{R}^2$ such that no two points of $P$ are on any concentric circle.  Then we have,
	\[ |\Delta(S,P)| \gtrsim \min(|S||P|^{1/4-\e},|S|^{2/3}|P|^{2/3},|S|^2,|P|^2).\]
\end{thm}

For comparison with the Theorems stated above, if $|S|=m$ and $|P|=n$ then we have

\[ 
|\Delta(S,P)| \gtrsim 
\begin{cases}
m^2, & m\lesssim n^{1/4-\e} \\
mn^{1/4-\e}, & n^{1/4-\e}\lesssim m\lesssim n^{5/4-\e} \\
m^{2/3}n^{2/3}, & n^{5/4-\e}\lesssim m \lesssim n^2 \\
n^2, & n^2\lesssim m.
\end{cases}
\]

Note that the the theorem would be false without the hypothesis that no two points in $P$ lie on a circle concentric with $S$ (though two could be replaced by another fixed constant).  To see this, let $S$ be a set of $n$ points on the unit circle, evenly spaced.  Let $P=\alpha S$ for some positive real number $\alpha$.  For a fixed $p\in P$,
\[
|\Delta(\{p\},S)|
\leq |S|=n
\]
But by symmetry, $\Delta(\{p\},S)=\Delta(\{p'\},S)$ for any $p,p'\in P$.  Thus, $|\Delta(P,S)|=|\Delta(\{p\},S)|\leq n$, which would contradict the conclusion of the theorem.  However, when the points are not evenly spaced as in the example, we obtain the following result.

\begin{thm}
\label{Thm: TwoCircles}
Let $S$ and $P$ be finite sets on concentric planar circles with center $O$, and suppose $\alpha$ is such that for any $\theta\in[-\pi,\pi]$ we have

\[
|\{(p,q)\in P^2:\angle pOq=\theta\}|\lesssim |P|^{2\alpha},
\]

where $\angle pOq$ is the oriented angle between line segments $\overline{Op}$ and $\overline{Oq}$.  Then,

\[
|\Delta(S,P)|\gtrsim |S|^{1/2}|P|^{1-\alpha}.
\]
\end{thm}

In particular, if no two points of $P$ subtend the same angle, then we get $|S|^{1/2}|P|$ distances.  One can check this is better than the bound in Theorem \ref{Thm: MainResult} regardless of the relative size of the sets. \\

We note that Theorems \ref{Thm: MainResult} and \ref{Thm: TwoCircles} are probably far from sharp. The only case we can find where two points sets $S$ and $P$ determine $o(|S||P|)$ distances are when $S$ and $P$ lie on concentric circles. We would be extremely interested in any examples between a point set $S$ contained to a circle and $P$ not on a concentric circle which determine  $o(|S||P|)$ distances.\\

It is interesting to compare our result to the other theorems above, since each answers the same question under different assumptions on the point sets.  Of particular interest is Mathialagan's result above, which makes no assumptions whatsoever on the (finite) point sets in question.  As mentioned above, our bound is better in the range $|P|^{1/2}\lesssim |S|\lesssim |P|^3$, but outside that range the bound $\Delta(P,S)\gtrapprox (|P||S|)^{1/2}$ wins, showing our structural assumptions on our sets are only helping when the sizes of the sets are (at least somewhat) balanced. \\

It is also worth comparing our result to that of Pach and de Zeeuw, who obtain a lower bound for the number of distinct distances determined by point sets each contained on a real algebraic curve of bounded degree.  We assume one of our sets lies on a circle (a much stronger assumption then a general curve) but the other is essentially arbitrary.  Their bound is never weaker than ours, and is strictly stronger in the balanced case. \\

This paper will be structured as follows.  In Section \ref{Sec: IncidenceProblem} we provide the initial framework for the bound on $|\Delta(S,P)|$, repeating the idea of Elekes that one can use pairs of repeated distances and we discuss the incidence bound we will use, due to Sharir and Zahl \cite{SZ17}. In Sections \ref{Sec: Behavior} and \ref{Sec: Proofs} we show that the hypotheses of the Sharir-Zahl incidence bound apply to our setting.

\subsection{Acknowledgements}
The authors wish to thank Adam Sheffer for introducing us to the problem, for pointing out the references \cite{M19} and \cite{SZ17}, and encouragement. The third author wishes to thank Adam Sheffer, Josh Zahl and the participants of the MSRI summer school on the Polynomial method for many helpful discussions and MSRI, Berkeley for hosting the workshop.

\section{Creating an Incidence Problem}\label{Sec: IncidenceProblem}

\begin{dfn}
	Given any two finite sets $S,P\subset\mathbb{R}^2$, define the distance set and quadruple set of $S$ and $P$, respectively, as
	
	\[
	\Delta(S,P)=\{|u-p|:u\in S,p\in P\},
	\]
	\[
	Q(S,P)=\{(u,v,p,q)\in S^2\times P^2:|u-p|=|v-q|\}.
	\]
\end{dfn}

\begin{thm}\label{Thm: CauchySchwarzEnergyBound}
	For any sets of points $S,P\subset\mathbb{R}^2$, we have
	
	\[
	|\Delta(S,P)|\geq\frac{|S|^2|P|^2}{|Q(S,P)|}
	\]
	
\end{thm}

\begin{proof}
	The statement follows directly from the classic Cauchy-Schwarz energy bound. Let $v(t) = \{(u,p) \in S \times P : |u-p|=t\}$ be the number of occurrences of the distance $t$, then we have
	
	$$ |S|^2|P|^2 = \left(\sum_{t\in \Delta(S,P)}v(t)\right)^2 \leq |\Delta(S,P)|\sum_t v^2(t) = |\Delta(S,P)|\cdot |Q(S,P)|. $$

\end{proof}

Therefore, an upper bound on the size of $Q(S,P)$ will yield a lower bound on the size of $\Delta(S,P)$.  In order to bound $Q(S,P)$ we will follow the approach of Pach and de Zeeuw \cite{PdZ17} and Bruner and Sharir \cite{BS18}, setting up an incidence problem.  \\

We start by making a couple simple reductions.  These are not additional hypotheses, but rather can be assumed without loss of generality.

\begin{itemize}
	\item Without loss of generality we may assume the circle $S$ lives on is the unit circle centered at the origin, so all $u\in S$ satisfy $\|u\|=1$.
	\item For technical reasons that will become apparent later, we want $P$ to have the property that distinct $p,q\in P$ always satisfy $\|p\|-\|q\|\neq 2$.  This can be achieved by considering half open annuli of width $2$, and assigning the annuli alternating colors.  Since one of the colors must contain at least half of $P$, we may assume $P$ has the desired property at the cost of a constant.
	\item We will assume $(-1,0)\notin S$ and $(0,0)\notin P$.
\end{itemize}

Note that for any point $p\in P$ and distance $t$, there are at most two choices of $u\in S$ for which $\|p-u\|=t$, since the circle centered at $p$ of radius $t$ can only intersect the unit circle twice. 
Therefore, the number of quadruples $(u,v,p,q)\in Q(S,P)$ with $p=q$ is $\approx |S||P|$.
It remains to bound our modified quadruple set

\[
\widetilde{Q}(S,P)=\{(u,v,p,q)\in S^2\times P^2:p\neq q,\|p-u\|=\|q-v\|\}.
\]

Let 

\[
f_{p,q}(u_1,u_2,v_1,v_2)=(p_1-u_1)^2+(p_2-u_2)^2-(q_1-v_1)^2-(q_2-v_2)^2 \label{Eqn: fpq}.
\]

It follows that for $u,v\in S$ and $p,q\in P$, we have $\|u-p\|=\|v-q\|$ if and only if $f_{p,q}(u,v)=0$.  This gives us an incidence problem between a set of points and hypersurfaces in $\R^4$, but we want to apply a point-curve incidence bound in $\R^2$.  To do this, we use the fact that $S$ is contained in the unit circle, which admits a rational parametrization.  Let

\[
\f(x)=\left(\frac{1-x^2}{1+x^2},\frac{2x}{1+x^2}\right).
\]

Then $\f$ is a homeomorphism from $\R$ to the unit circle with the point $(-1,0)$ removed.  Recall $f_{p,q}$ is a quadratic in four variables; this means

\[
F_{p,q}(x,y):=(1+x^2)^2(1+y^2)^2f_{p,q}(\f(x),\f(y))
\]

is a two variable polynomial of degree $\leq 12$.  Moreover, $F_{p,q}(x,y)=0$ if and only if $f_{p,q}(\f(x),\f(y))=0$, which in turn happens if and only if $(x,y)$ parametrizes a point $(u,v)\in S^1\times S^1$ with $\|u-p\|=\|v-q\|$.  If $\Pi=\{(x,y):\f(x),\f(y)\in S\}$ and $\Gamma=\{Z(F_{p,q}):p,q\in P,p\neq q\}$, then our observations can be summarized in the following lemma.

\begin{lem} \label{Lem: IncidenceProblemReduction}
	If $S,P,\Pi,\Gamma$ are as above, then $|\Pi|=|S|^2,|\Gamma|\approx |P|^2$, and
	
	\[
	Q(S,P)\approx |S||P|+I(\Pi,\Gamma).
	\]
\end{lem}

So, matters have been reduced to obtaining an incidence bound that applies to our sets $\Pi$ and $\Gamma$. \\

The incidence bound we will use is due to Sharir and Zahl \cite{SZ17}. To use the result of Sharir and Zahl we will introduce their terminology.  We identify polynomials of degree $\leq D$ with elements of $\R^{{\tiny\binom{D+2}{2}}}$, since each polynomial can be viewed as the vector of its coefficients.  Since $Z(f)=Z(\lambda f)$ for nonzero scalars $\lambda$, we can identify algebraic plane curves of degree $\leq D$ as elements of $\PRD$.  With this framework, Sharir and Zahl make the following definition.

\begin{dfn}
	An $s$ dimenional family of plane curves of degree at most $D$ is an algebraic variety $\mathcal{F}\subset\PRD$ that has dimension $s$. We call the degree of $\mathcal{F}$ the complexity of the family.
\end{dfn}

\begin{thm}[\cite{SZ17}, Theorem 1.3]\label{Thm: SharirZahl}

Let $\mathcal{C}$ be a set of algebraic plane curves belonging to an $s$-dimensional family of curves of bounded degree, no two of which share a common irreducible component.  Let $\mathcal{P}$ be a finite set of points in $\R^2$.  For any $\e>0$, we have

\[
I(\mathcal{P},\mathcal{C})\lesssim |\mathcal{P}|^{\frac{2s}{5s-4}}|\mathcal{C}|^{\frac{5s-6}{5s-4}+\varepsilon}+|\mathcal{P}|^{\frac{2}{3}}|\mathcal{C}|^{\frac{2}{3}}+|\mathcal{P}|+|\mathcal{C}|,
\]

where the constant depends on $s$, $\varepsilon$, $D$ and the complexity of the family that $\mathcal{C}$ is selected from.
	
\end{thm}

In Section \ref{Sec: Proofs}, we will prove that $\Gamma$ lives in a $4$-dimensional family of algebraic plane curves, and that no two curves of $\Gamma$ have a common component.  In order to prove this, we will first work with related curves in $\R^4$ before applying the rational parametrization to obtain curves in the plane.  Recall that the polynomial $f_{p,q}$ defines an algebraic hypersurface in $\R^4$.  We define $C_{p,q}=Z(f_{p,q})\cap (S^1\times S^1)$.  Since $S^1\times S^1$ is a real algebraic variety of dimension $2$ and $f_{p,q}$ cannot vanish on all of $S^1\times S^1$ (unless $p=q=0$), it follows that $C_{p,q}$ is a real algebraic curve in $\R^4$.  We will study these curves in Section \ref{Sec: Behavior}.

\section{Behavior of the curves $C_{p,q}$ in $\R^4$} \label{Sec: Behavior}

In this section, we will prove some results about our curves $C_{p,q}\subset\R^4$.  These properties will be used in section \ref{Sec: Proofs} to prove that no two curves in $\Gamma$ share a common component.

\begin{lem}\label{Lem: DimensionOne}
For any $p\neq q$, the curve $C_{p,q}$ does not contain any isolated points (in the Euclidean topology).
\end{lem}

\begin{proof}
Consider some $(u,v)\in C_{p,q}$.  Since $u,v\in S^1$, $||p||-1\leq ||u-p||\leq ||p||+1$ and similarly $||q||-1\leq ||v-q||\leq ||q||+1$.  Moreover, these inequalities are all strict since otherwise one of $\|p\|=\|q\|$ or $\|p\|=\|q\|+2$ must hold.  It follows that for some $\e>0$ and for any $t\in (-\e,\e)$ there must be points $u_t,v_t\in S^1$ with $||u_t-p||=||u-p||+t$ and $||v_t-q||=||v-q||+t$, hence $(u_t,v_t)\in C_{p,q}$.  We may also require that $u,u_t$ are on the same side of the circle in the sense that $S^1$ is divided into two semi-circles by the line through $p$ and the origin.  Making a similar restriction for $v_t$, this ensures that $t\mapsto (u_t,v_t)$ is continuous, hence $(u,v)\in C_{p,q}$ is not an isolated point.

\end{proof}

\begin{lem}\label{Lem: 2flatIntersectionBound}
For any $p,q\in \R^2$ with $\|p\|\neq\|q\|$ and any $2$-flat $K$, we have $|C_{p,q}\cap K|\leq 4$.
\end{lem}

\begin{proof}
Let $(u_1,v_1),(u_2,v_2),(u_3,v_3)\in C_{p,q}\cap K$.  It follows that $K-(u_3,v_3)$ is the plane spanned by $\{(u_1,v_1)-(u_3,v_3),(u_2,v_2)-(u_3,v_3)\}$.  If $(u,v)\in C_{p,q}\cap K$ then $(u,v)-(u_3,v_3)$ is in that plane as well, and therefore

\[
(u,v)-(u_3,v_3)=x((u_1,v_2)-(u_3,v_3))+y((u_2,v_2)-(u_3,v_3))
\]

for some unique values of $x$ and $y$.  Equivalently, we have

\begin{align*}
u=&x(u_1-u_3)+y(u_2-u_3)+u_3 \\
v=&x(v_1-v_3)+y(v_2-v_3)+v_3
\end{align*}

Since $\|u\|=\|v\|=1$, this means $x$ and $y$ satisfy

\begin{align*}
\|x(u_1-u_3)+y(u_2-u_3)+u_3\|&=1 \\
\|x(v_1-v_3)+y(v_2-v_3)+v_3\|&=1.
\end{align*}

This system has solutions $(x,y)=(1,0),(0,1),(0,0)$ corresponding to $(u_1,v_1),(u_2,v_2),(u_3,v_3)$, respectively.  To prove the lemma, it therefore suffices to show this system has $\leq 4$ solutions.  For now we assume that both equations are irreducible quadratics, dealing with the case that these are lines or the product of lines later.  If both equations are irreducible then by Bezout it suffices to prove that one is not a constant multiple of the other.  Expanding each equation and focusing on the quadratic terms while ignoring the lower order terms, we have

\begin{align*}
\|u_1-u_3\|^2x^2+\|u_2-u_3\|^2y^2+2\left\langle u_1-u_3,u_2-u_3\right\rangle xy+\cdots &=0\\
\|v_1-v_3\|^2x^2+\|v_2-v_3\|^2y^2+2\left\langle v_1-v_3,v_2-v_3\right\rangle xy+\cdots &=0.
\end{align*}

Suppose for contradiction that one equation is a constant multiple of the other.  This means we can normalize both equations so that the coefficient of $x^2$ is 1, and all other coefficients must be the same.  In particular, this means we have $\frac{\|u_2-u_3\|}{\|u_1-u_3\|}=\frac{\|v_2-v_3\|}{\|v_1-v_3\|}$; denote this common value by $t$.  Let $A=\|u_1-u_3\|,B=\|v_1-v_3\|$, define $\theta_u$ to be the angle between $u_1-u_3$ and $u_2-u_3$, and define $\theta_v$ similarly.  Then our equations are

\begin{align*}
A^2x^2+t^2A^2y^2+(2tA^2\cos\theta_u) xy+\cdots &=0\\
B^2x^2+t^2B^2y^2+(2tB^2\cos\theta_v) xy+\cdots &=0
\end{align*}

or, normalizing so that the $x^2$ coefficient is 1,

\begin{align*}
x^2+t^2y^2+(2t\cos\theta_u) xy+\cdots &=0\\
x^2+t^2y^2+(2t\cos\theta_v) xy+\cdots &=0.
\end{align*}

Comparing the $xy$ coefficients, we conclude that $\theta_u=\pm\theta_v$.  This means that the two triangles $\Delta u_1u_2u_3$ and $\Delta v_1v_2v_3$ are similar; they have common angle $\theta$ at the third vertex and side lengths of the form $\ell,t\ell$ from the third vertex to the first and second, respectively, for some value of $\ell$ ($\ell=A$ in the first triangle, and $\ell=B$ in the second).  We claim there is only one value of $\ell$ for which such a triangle has all its vertices on the unit circle.  This in turn implies $\|p\|=\|q\|$, as the triangles $\Delta u_1u_3p$ and $\Delta v_1v_3q$ would be congruent, and this is our contradiction.  So, it suffices to prove the claim.  Let $O$ denote the origin, and let $\Delta \alpha\beta\gamma$ be any triangle with angle $\theta$ at $\alpha$ and side lengths $\alpha\beta=t\ell,\alpha\gamma=\ell$  (see figure 1).

\begin{figure}
\centering
\begin{tikzpicture}
\draw (0,0)--(-4.5,2.18)--(-1.5,4.77)--(3,4)--(0,0);
\draw (0,0)--(-1.5,4.77);
\draw[ultra thick] (-4.5,2.18)--(3,4);
\node[below] at (0,0) {$O$};
\node[above] at (-4.5,2.18) {$\gamma$};
\node[above] at (-1.5,4.77) {$\beta$};
\node[above] at (3,4) {$\alpha$};
\node[above] at (-.75,3.09) {\textbf{$\ell$}};
\node[above] at (.75,4.38) {$t\ell$};
\node[left] at (1.8,3.95) {$\theta$};

\draw (1.875, 4.1925)
 arc [radius=1, start angle=170.29, end angle= 196.21];
\end{tikzpicture}
\caption{}
\end{figure}

If $\alpha,\beta,\gamma$ are on the unit circle then the triangle $\Delta O\alpha\gamma$ is an isosceles triangle with common side length $1$ and base length $\ell$, hence it has common base angle $\angle O\alpha\gamma=\arccos \ell/2$.  This implies $\angle O\alpha\beta=\theta+\arccos\ell/2$.  Similarly, $\Delta O\alpha\beta$ is an isosceles triangle with common side length $1$ and base length $t\ell$, so the common base angle is $\angle O\alpha\beta=\arccos t\ell/2$.  This means $\ell$ must satisfy

\begin{align*}
\theta+\arccos \frac{\ell}{2}&=\arccos \frac{t\ell}{2} \\
\frac{\ell}{2}\cos\theta-\left(1+\frac{\ell^2}{4}\right)^{1/2}\sin\theta&=\frac{t\ell}{2} \\
\left(\frac{\cos\theta-t}{2}\right)^2\ell^2&=\sin^2\theta\left(1+\frac{\ell^2}{4}\right) \\
\left(\left(\frac{\cos\theta-t}{2}\right)^2-\frac{\sin^2\theta}{4}\right)\ell^2&=\sin^2\theta.
\end{align*}

Note the right hand side cannot be zero, since that would imply three points on a circle were also on a line.  So, there is at most one positive solution for $\ell$. \\

We now show that the equations are irreducible quadratics.  Recall the first few terms are

\[
x^2+t^2y^2+(2t\cos\theta) xy+\cdots =0
\]

If this is reducible, we can write it as a product $(x+ay)(x+by+c)$.  Expanding, we must have

$$ ab = t^2 \text{~~~and~~~} a+b=2t\cos(\theta).$$

Plugging the first of these into the second gives the quadratic

$$  \frac{1}{a}t^2 - 2\cos(\theta)t + a = 0,$$

which has a real solution only when $\cos(\theta)=\pm 1$. As noted above, this would mean that we have three points of a circle on a line, a contradiction. So we are not in the case where are quadratics are reducible and thus we are done by the above.

\end{proof}

\begin{lem}
\label{Lem: CurveIntersection}
If $(p,q),(p',q')\in \R^4$ are distinct and $\|p\|\neq\|q\|$, then $|C_{p,q}\cap C_{p',q'}|\leq 4$.
\end{lem}

\begin{proof}
The curves $C_{p,q}$ are defined by the equations

\begin{align}
u_1^2+u_2^2&=1 \\
v_1^2+v_2^2&=1 \\
(u_1-p_1)^2+(u_2-p_2)^2&=(v_1-q_1)^2-(v_2-q_2)^2.
\end{align}

Expanding equation (3) and substituting equations (1) and (2), we see that (3) can be replaced by

\[
\tag{$3'$}
N_{p,q}\cdot (u,v)=\frac{\|p\|^2-\|q\|^2}{2},
\]

where $N_{p,q}=(p_1,p_2,-q_1,-q_2)$.  It follows that for any scalar $t\neq 0,1$, $(u,v)\in C_{tp,tq}$ must satisfy

\[
N_{tp,tq}\cdot (u,v)=t^2\frac{\|p\|^2-\|q\|^2}{2},
\]

or

\[
N_{p,q}\cdot (u,v)=t\frac{\|p\|^2-\|q\|^2}{2}.
\]

This means that if $(p,q)$ and $(p',q')$ are distinct scalar multiples of each other, $C_{p,q}$ and $C_{p',q'}$ are contained in disjoint hyperplanes and thus have empty intersection.  It remains to consider the case where $(p,q),(p',q')$ are distinct but not scalar multiplies of each other.  In this case, $N_{p,q}$ and $N_{p',q'}$ are not scalar multiples of each other either, and hence $C_{p,q},C_{p',q'}$ are contained in the intersection of two hyperplanes, which is a $2$-flat.  By Lemma \ref{Lem: 2flatIntersectionBound}, any such $2$-flat can contain at most $4$ points from $C_{p,q}$.

\end{proof}

\section{Proofs of Theorems \ref{Thm: MainResult} and \ref{Thm: TwoCircles}}
\label{Sec: Proofs}

\subsection{Proof of Theorem \ref{Thm: MainResult}}
We are now ready to prove that our curves satisfy the hypotheses of the Sharir-Zahl incidence bound (Theorem \ref{Thm: SharirZahl}).  We must show that no two curves in $\Gamma$ share a common irreducible component, and that $\Gamma$ belongs to a $4$-dimensional family.  \\

To prove no two curves of $\Gamma$ share a common component, we use our work in Section \ref{Sec: Behavior}.  By Lemma \ref{Lem: CurveIntersection} and the injectivity of $\f$, the intersection of any two curves of $\Gamma$ is finite.  Therefore, any common component must be zero dimensional, and hence an isolated point $(x_0,y_0)$.  Since $\f\times\f$ maps $Z(F_{p,q})\to C_{p,q}$, we may consider the point $(u_0,v_0):=(\f(x_0),\f(y_0))\in C_{p,q}$.  By Lemma \ref{Lem: DimensionOne}, $(u_0,v_0)$ is not isolated.  This contradicts the continuity of $\f^{-1}$. \\

To prove $\Gamma$ is contained in a $4$-dimensional family, we will work over the field of complex numbers temporarily. Thus, we let $p_1,p_2,q_1$, and $ q_2$ vary over $\C$ instead of $\R$. we observe that the coefficients of (complexified) curves in $\Gamma$ are polynomials in our parameters $p_1,p_2,q_1,q_2$, and these polynomials do not simultaneously vanish unless $p=q=0$. This map $\C^4 \setminus\{ 0\} \to \C^{\binom{D+2}{2}}$ is thus clearly a morphism of quasi-projective varieties. The natural surjection $\C^{\binom{D+2}{2}} \setminus\{0\}\surjects \PCD$ is also a morphism of quasi-projective varieties, and hence so is the composed map $\C^4 \setminus\{0\} \to \PCD$. 

Thus, $\Gamma$ is contained in the image of a morphism of quasi-projective varieties $\C^4\setminus\{0\}\to \PCD$. After taking the Zariski closure of the image, if necessary, we see that $\Gamma$ is contained in a variety in $\PCD$ of dimension $\leq 4$, by invoking, for example \cite[Theorem 11.12]{H13book}. Thus, now restricting ourselves to $\R$ and using real dimension, $\Gamma$ is contained in an family of dimension $\leq 4$ (since the Sharir-Zahl incidence bound gets worse as $s$ increases, this is enough). \\

Now that we have established that Theorem \ref{Thm: SharirZahl} applies to our curves, we are now ready to complete the proof of Theorm \ref{Thm: MainResult}.  Applying Theorem \ref{Thm: SharirZahl} with $s=4$, we get

\[
I(\Pi,\Gamma)\lesssim |\Pi|^{1/2}|\Gamma|^{7/8+\e}+|\Pi|^{2/3}|\Gamma|^{2/3}+|\Pi|+|\Gamma|.
\]

By Lemma \ref{Lem: IncidenceProblemReduction}, we have

\[
Q(S,P)\lesssim |S||P|^{7/4+\e}+|S|^{4/3}|P|^{4/3}+|S|^2+|P|^2.
\]

By Theorem \ref{Thm: CauchySchwarzEnergyBound}, this gives

\[
\Delta(S,P)\gtrsim \min(|S||P|^{1/4-\e},|S|^{2/3}|P|^{2/3},|S|^2,|P|^2),
\]

as claimed.

\subsection{Proof of Theorem \ref{Thm: TwoCircles}}

We use the following well known theorem from additive combinatorics.

\begin{thm}[Ruzsa's Triangle Inequality]
Let $A,B,C$ be finite subsets of an abelian group.  Then,

\[
|A||B-C|\leq |A-B||A-C|.
\]
\end{thm}

To prove Theorem \ref{Thm: TwoCircles}, we reduce matters to counting difference sets of angles.  We first observe there is a line $\ell$ passing through $O$ with the property that one side of $\ell$ contains a positive proportion of both $S$ and $P$.  To prove this, first note that we may assume that $P$ is on either the upper or lower semicircle by throwing away up to half of $P$.  We can then choose a point on the semicircle such that the remaining points in $P$ are evenly divided to the left and right.  Let $\ell$ be the line through that point and $O$.  By construction, both sides of $\ell$ contain at least $\frac{1}{4}|P|$ points of $P$.  Since one side must contain $\frac{1}{2}|S|$ points of $S$.  So, $\ell$ has the desired property. \\

For the remainder of the proof, we will assume both sets are contained entirely on one side of $\ell$.  For all $p$ in either set, let $\theta_p$ be the angle from $\ell$ to the line segment $\overline{Op}$, so $\theta_p\in [0,\pi]$ for all $p\in S,P$.  We also observe that $\angle pOq=\theta_p-\theta_q\in [-\pi,\pi]$.  Let $A=\{\theta_p:p\in S\}$ and $B=\{\theta_p:p\in P\}$. If the circles containing the sets have radii $r_1,r_2$, then for any $u\in S,p\in P$ we have

\[
\|p-u\|=\left\langle p-u,p-u\right\rangle =r_1^2+r_2^2-2r_1r_2\cos(\theta_p-\theta_u),
\]

hence $|\Delta(S,P)|\gtrsim|A-B|$ as $\cos$ is $2$-to-one on $[-\pi,\pi]$.  By assumption, the map $P^2\to[-\pi,\pi]$ given by $(p,q)\mapsto \theta_p-\theta_q$ is $\lesssim |P|^{2\alpha}$-to-one, so $|B-B|\gtrsim |B|^{2-2\alpha}$.  Applying Ruzsa's triangle inequality with $C=B$, we get

\[
|A||B-B|\leq |A-B|^2,
\]

or 

\[
|A|^{1/2}|B|^{1-\alpha}\leq |A-B|
\]

as claimed.

\bibliography{DiscreteGeometryReferences}{}
\bibliographystyle{plain}

	\bigskip
		{\footnotesize
		
		 \textit{E-mail address}, A.~McDonald: \texttt{a.mcdonald@rochester.edu}\\
		 \textit{E-mail address}, B.~McDonald: \texttt{bmcdon11@ur.rochester.edu}\\
		 \textit{E-mail address}, J.~Passant: \texttt{jpassant@ur.rochester.edu}\\
		 \textit{E-mail address}, A.~Sahay: \texttt{asahay@ur.rochester.edu}\\
		
		\textsc{Department of Mathematics, University of Rochester, Rochester, NY 14627}\par\nopagebreak}

\end{document}